\documentclass[12pt,a4paper]{amsart}

\usepackage{docmute}
\usepackage{a4wide}
\usepackage[utf8]{inputenc}
\usepackage{amsmath,amssymb}

\usepackage[
backend=biber,
style=numeric,
sorting=nyt
]{biblatex}
\usepackage{xcolor}
\usepackage{booktabs}
\usepackage{hyperref}
\hypersetup{
	colorlinks,
	citecolor=black,
	filecolor=black,
	linkcolor=blue,
	urlcolor=black
}
\usepackage{stmaryrd}
\usepackage{url}
\usepackage{longtable}
\usepackage[figuresright]{rotating}
\usepackage{amsthm}
\usepackage{bbold}
\usepackage{enumerate}
\usepackage{geometry}
\newtheorem{definition}{Definition}
\newtheorem{theorem}[definition]{Theorem}

\newtheorem{lemma}[definition]{Lemma}

\newtheorem{corollary}[definition]{Corollary}

\newtheorem{problem}{Problem}

\newtheorem*{claim*}{Claim}

\newcommand{\0}{\emptyset}
\newcommand{\mc}{\mathcal}
\newcommand{\mbb}{\mathbb}

\newcommand{\foralmostall}{\forall^\infty}
\newcommand{\existinfty}{\exists^\infty}
\newcommand{\Ee}{\mc{E}}
\newcommand{\Ii}{\mc{I}}

\newcommand{\Mm}{\mc{M}}
\newcommand{\Nn}{\mc{N}}

\newcommand{\bez}{\backslash}
\newcommand{\se}{\subseteq}

\newcommand{\es}{\supseteq}

\newcommand{\rest}{\hspace{-0.25em}\upharpoonright\hspace{-0.25em}}

\newcommand{\concat}{^{\frown}}
\newcommand{\cons}{^{\frown}}

\newcommand{\tn}[1]{\textnormal{#1}}
\newcommand{\ti}[1]{\textit{#1}}

\newcommand{\dom}{\textnormal{dom}}

\newcommand{\splitt}{\textnormal{split}}

\newcommand{\succe}{\textnormal{succ}}

\def\c{\mathfrak{c}}

\def\w{\omega}

\def\iff{\longleftrightarrow}

\title{On algebraic sums, trees and ideals in the Cantor space}

\author{Marcin Michalski}
\email{marcin.k.michalski@pwr.edu.pl}

\author{Robert Rałowski}
\email{robert.ralowski@pwr.edu.pl}

\author{Szymon Żeberski}
\email{szymon.zeberski@pwr.edu.pl}

\thanks{The work has been partially financed by grant {\bf 8211204601, MPK: 9130730000} from the Faculty of Pure and Applied Mathematics, Wrocław University of Science and Technology.
	}

\address{Marcin Michalski, Robert Rałowski, Szymon Żeberski, Faculty of Pure and Applied Mathematics, Wrocław University of Science and Technology, 50-370 Wrocław, Poland}

\date{}

\begin{document}

\begin{abstract}
	We work in the Cantor space $2^\w$. The results of the paper adhere the following pattern. Let $\Ii\in \{\Mm, \Nn, \Mm\cap \Nn, \Ee\}$ and $T$ be a perfect, uniformly perfect or Silver tree. Then for every $A\in \Ii$ there exists $T'\se T$ of the same kind as $T$ such that $A+\underbrace{[T']+[T']+\dots +[T']}_{\text{n--times}}\in \Ii$ for each $n\in\w$. We also prove weaker statements for splitting trees. For the case $\Ee$ we also provide a simple characterization of basis of $\Ee$. We use these results to prove that the algebraic sum of a generalized Luzin set and a generalized Sierpiński set belongs to $u_0$ and $v_0$, provided that $\c$ is a regular cardinal.
\end{abstract}

\maketitle

\section{Introduction and notation}
 
  We adopt the standard set-theoretical notation (see e.g. \cite{Jech}). Throughout the paper $\Mm$ will denote the $\sigma-$ideal of meager subsets of the Cantor space $2^\w$, $\Nn$ will denote the $\sigma-$ideal of null subsets of $2^\w$ and $\Ee$ will denote the $\sigma-$ideal of sets generated by closed null subsets of $2^\w$.
  
  Let $T\se 2^{<\omega}$ be a tree. We will use the following notions related to trees:
  \begin{itemize}
	  \item $\succe_T(\sigma)=\{i\in 2: \sigma\concat i\in T\}$;
  	\item $\splitt(T)=\{\sigma\in T: |\succe_T(\sigma)|= 2\}$.
  \end{itemize}
  
  Let us recall definitions of some families of trees.
  \begin{definition}
    A tree $T\se 2^{<\w}$ is called
    \begin{itemize}
      \item a Sacks or perfect tree, if $(\forall \sigma\in T) (\exists \tau\in T) (\sigma\se\tau \land \tau\in \splitt(T))$;
      \item a uniformly perfect, if for every $n\in\omega$ either $2^n\cap T\se\splitt(T)$ or $2^n\cap \splitt(T)=\0$;
      \item\label{def Silver} a Silver tree if it is perfect and
      \[
      (\forall \sigma, \tau \in T)(|\sigma|=|\tau| \to (\forall i\in\{0,1\})(\sigma\cons i \in T \iff \tau\cons i \in T));
      \]
	    \item a splitting tree, if
	    \[
	      (\forall \sigma\in T) (\exists N\in\w) (\forall n>N) (\forall i\in 2) (\exists \tau\in T) (\sigma\se \tau \land \tau(n)=i).
	    \] 
    \end{itemize}
  \end{definition}

  Notice that for a Silver tree $T$ there exist $x_T\in 2^\omega$ and $A\in[\omega]^\omega$ such that
	\[
	  (\forall \sigma\in 2^{<\omega})(\sigma \in T \iff (\forall n\in\dom(\sigma)\bez A)(\sigma(n)=x_T(n)).
	\]
	In such a case we will call $A$ and $x_T$ the parameters of $T$.
  
  $[T]$ will denote the body of $T$ - the set of infinite branches of $T$, i.e.
  \[
    [T]=\{x\in 2^\w:\; (\forall n\in\w)(x\rest n \in T)\}.
  \]

  Let $+$ be a coordinate wise addition modulo $2$ on $2^\w$. For sets $A, B\se 2^\w$ the algebraic sum of these sets is defined as follows
  \[
    A+B=\{a+b\in 2^\w:\;a\in A \land b\in B\}.
  \]
  We will extend the same notation to express and $x+y$ and $C+x$ for and $x\in 2^X, y\in 2^Y$ and $C\se 2^Z$ where $X,Y, Z\se \w$ and $C\se 2^Z$ in a natural way, i.e.
  \begin{align*}
    x+y&=x\rest (X\bez Y)\cup \{(n,x(n)+y(n)):\; n\in X\cap Y\}\cup y\rest (Y\bez X),
    \\
    C+x&=\{c+x: c\in C\}.
  \end{align*}
   
  Occasionally we will use the following notation for the sequence of zeros: $\mbb{0}=(0, 0, \dots)$; and the sequence of ones: $\mbb{1}=(1, 1, \dots)$.
	
	Notice that for Silver trees $T_1$ and $T_2$ with associated parameters $A_1, x_{T_1}$ and $A_2, x_{T_2}$ respectively, the set $[T_1]+[T_2]$ is a body of a Silver tree with parameters $A_1\cup A_2$ and $x_{T_1}+x_{T_2}$. In particular if $T$ is a Silver tree with parameters $x_T$ and $A$, then $[T]+[T]$ has parameters $\mbb{0}$, $A$. It follows that for any $n\in\omega$
	\[
		\underbrace{[T]+[T]+...+[T]}_{n\ti{--times}}\se [T]\cup ([T]+[T]).
	 \]
	This fact is the reason why the proofs of Theorems \ref{meager Silver}, \ref{Null Silver}, \ref{Silver E} are the least difficult compared to analogous results concerning other types of trees.
	  
	Let $\mbb{T}$ be a family of trees. For such families we define the tree ideal $t_0$ as follows 
	\[
	  A\in t_0 \iff (\forall T\in \mbb{T})(\exists T'\se T, T'\in\mbb{T})([T']\cap A=\0).
	\]
  
  Examples of such tree ideals include $s_0$, $u_0$ and $v_0$ for the family of Sacks, uniformly perfect and Silver trees respectively. The $\sigma$-ideal $s_0$ was the first of its kind to be studied (within the context of the real line) and is also known as the Marczewski ideal. 
  
  Let us recall the notion of $\Ii-$Luzin sets.
  \begin{definition}
    Let $\Ii$ be a $\sigma-$ideal. We call a set $L$ an $\Ii-$Luzin set, if $|L\cap A|<|L|$ for every $A\in \Ii$.
  \end{definition}
  $\Mm-$Luzin sets are called generalized Luzin sets and $\Nn-$Luzin sets are called generalized Sierpiński sets. Regarding the notions of $\Ii-$Luzin sets, tree ideals and algebraic sums, in \cite[Theorem 2.12]{MiZeb} and \cite[Theorem 27]{MiRalZebNon} the authors proved some results that can be combined in the following theorem.
  \begin{theorem}\label{Luziny t0}
    Let $\c$ be a regular cardinal. Then for every generalized Luzin set $L$ and generalized Sierpiński set $S$ we have $L+S\in s_0\cap m_0\cap l_0$.
  \end{theorem}
  Here $m_0$, $l_0$ are tree ideals connected with the families of Miller trees and Laver trees. The key ingredients of the proofs of these theorems were results concerned with algebraic sums of perfect sets with sets from $\sigma-$ideals of meager sets or null sets. In this paper, while we generalize these results, we also obtain the analogue of the above theorem for the case of $u_0$ and $v_0$ in $2^\w$. For the similar purpose algebraic sums were studied in \cite{NoScheeWeiss} and \cite{Scheepers}. In various contexts they were also studied in \cite{ErdKuMa}. Results related to the ones presented in this paper were also helpful in \cite{Rec}.
	
	\section{Trees and algebraic sums}

	\subsection*{Category case}

	  Let us recall following characterization of meager sets in $2^\omega$ from \cite[Theorem 2.2.4]{BarJu}.
		\begin{lemma}\label{baza meagery}
			Let $F$ be a meager subset of $2^\omega$. There is $x_F\in 2^\omega$ and a partition $\{I_n: n\in\omega\}$ of $\omega$ into intervals such that
			\[
				F\se \{x\in 2^\omega: (\foralmostall n)(x\rest I_n \neq x_F\rest I_n)\}.
			\]
		\end{lemma}
	  
	  \begin{theorem}\label{meager Silver}
			For every meager set $F\se 2^\omega$ and every Silver tree $T$ there is a Silver tree $T'\se T$ such that for every $n\in\omega$
			\[
				F+\underbrace{[T']+[T']+...+[T']}_{n\ti{--times}}\in\Mm.
			\]
		\end{theorem}
		\begin{proof}
		  Let $F$ be meager with $x_F$ and $(I_n:\; n\in\omega)$ be as in the Lemma \ref{baza meagery} and let $T$ be a Silver tree with $x_T$ and $A$ as per characterization of Silver trees \ref{def Silver}. Let $J_n=I_{2n}\cup I_{2n+1}$ and $A'\se A$ such that $|A|=\aleph_0$ and $|A'\cap J_n|\leq 1$. Let 
		  \[
		    x_{F_n}=x_F+\underbrace{x_T+x_T+\dots+x_T}_{n\ti{--times}}
		  \]
		  for each $n\in\omega$. For each $n\in\omega$ set
		  \[
		    F_n=\{x\in 2^\omega:\; (\foralmostall k)(x\rest J_k \ne x_{F_n}\rest J_k)\}
		  \]
		\end{proof}
		
		\begin{theorem}
			For every meager set $F\se 2^\omega$ and every (uniformly) perfect tree $T$ there is a (uniformly) perfect tree $T'\se T$ such that for every $n\in\omega$
			\[
				F+\underbrace{[T']+[T']+...+[T']}_{n\ti{--times}}\in\Mm.
			\]
		\end{theorem}
		\begin{proof}
		  Let us focus on the case of perfect trees, the proof for uniformly perfect trees is almost identical. So, let $T$ be a perfect tree and $F$ be a meager set with associated $x_F\in 2^\omega$ and $(I_n:\; n\in\omega)$ according to the Lemma \ref{baza meagery}. Let $k_{-1}=-1$, $k_{n+1}=k_n+(2^{n+1})^{n+2}$ and set
		  \begin{align*}
		    \textstyle J_0=I_0,\quad J_{n+1}=\bigcup_{k=k_{n}+1}^{k_{n+1}}I_k.
		  \end{align*}
		  Let $\sigma_{\0}'$ be a splitting node of $T$ such that $J_0\se \sigma_\0'$ and for each $\tau\in 2^n$ and $i=0,1$ let $\sigma_{\tau\cons i}'$ be a splitting node of $T$ such that ${\sigma_{\tau}'}\cons i\se \sigma_{\tau\cons i}$ and $\bigcup_{k=0}^{n+1}J_k\se \dom(\sigma_{\tau\cons i}')$. We use a surjection $\{1,2,\dots, 2^{n(n+1)}\}\ni i\mapsto (\tau_0^i, \tau_1^i, \dots, \tau_{m_i}^i)\in (\{0,1\}^n)^{\leq n}$ to set
		  \[
		    x_H\rest I_{k_{n-1}+i}=\left( x_F+\sum_{j=0}^{m_i}\sigma'_{\tau^i_{j}}\right) \rest I_{k_{n-1}+i}.
		  \]
		  The construction is complete. Finally, set $T'=\{\sigma\in 2^\omega:\; (\exists \tau\in 2^{<\omega})(\sigma\se \sigma'_\tau)\}$. It follows from the construction that for $H=\{x\in 2^\omega:\; (\foralmostall n)(x\rest J_n\neq x_H\rest J_n)\}$ satisfies
		  \[
		    F+\underbrace{[T']+[T']+\dots [T']}_{n\ti{--times}}\se H
		  \]
		  for each $n\in\omega$.
		\end{proof}
	  The following result concerned with splitting trees has a slightly weaker thesis.
	  \begin{theorem}\label{meager splitting}
			For every meager set $F\se 2^\omega$ there is a splitting tree $T$ such that for every $n\in\omega$
			\[
				F+\underbrace{[T]+[T]+...+[T]}_{n\ti{--times}}\in\Mm.
			\]
		\end{theorem}
		\begin{proof} 
			Let $F\se 2^{\omega}$ be a meager set and assume there is $x_F\in 2^\omega$ and a partition of $\omega$ into intervals $\{I_n: n\in\omega\}$, as in Lemma \ref{baza meagery}.
			
		  Let $x_{F'}=x_F$ and let $I'_0=I_0$ and
			\[
				I'_{n+1}=\bigcup\{I_k: k\in [\frac{n(n+1)}{2}+1,\frac{(n+1)(n+2)}{2}+1)\}
			\]
			for $n>0$. Define 
			\[
				B_0=\{\tau\in 2^{I'_0}: (\exists n\in I'_0)(\forall k\in I'_0\bez\{n\})(\tau(n)=1 \land \tau(k)=0)\}
			\]
			and for any  $n\in\omega$
			\begin{align*}
				B_{n+1}=&\{\tau\cons \sigma: \tau\in B_n\; \land\; \sigma\in 2^{|I'_{n+1}|}\;\land
				\\
				&\;\land\;(\exists m<|I'_{n+1}|)(\forall k< |I'_{n+1}|\bez\{m\})(\sigma(m)=1 \land \sigma(k)=0)\}.
			\end{align*}
			Set $B=\bigcup_{n\in\omega}B_n$ and $T=\{\tau\in 2^{<\omega}: (\exists \sigma\in B)(\tau\se\sigma)\}$. Clearly $T$ is a splitting tree. To see that the rest of the thesis holds, let $n\in\omega$ and take $x\in F+\underbrace{[T]+[T]+...+[T]}_{\tn{n-times}}$. Then there are $x_0\in F$ and $x_1, x_2, ..., x_n\in [T]$ such that $x=x_0+x_1+...+x_n$. Consider $I'_{n+k}$ for $k>0$ large enough so that $x_0\rest I'_{n+k}\neq x_{F'}\rest I'_{n+k}$. Since $I'_{n+k}$ is a union of $n+k$ intervals from the original partition, there is $n_0\in\omega$ with $I_{n_0}\se I'_{n+k}$ and $x_j\rest I_{n_0}=\mathbb{0}$ for every $j\in\{1,2,...,n\}$. Then $x_0+x_1+...+x_n\rest I'_{n+k}\neq x_{F'}\rest I'_{n+k}$. It follows that
			\[
				F+\underbrace{[T]+[T]+...+[T]}_{\tn{n-times}}\se\{x\in 2^\omega: (\foralmostall n)(x\rest I'_n\neq x_{F'}\rest I'_n)\}\in\Mm.
			\]
		\end{proof}
		
		The natural question remains if one can prove the following statement regarding splitting trees.
		\begin{problem}
		  Let $F$ be a meager subset of $2^\w$. Is it true that for every splitting tree $T$ there exists a splitting tree $T'\se T$ such that
		  \[
		    F+\underbrace{[T']+[T']+...+[T']}_{\tn{n-times}}\in\Mm?
		  \]
		\end{problem}
	  
	\subsection*{Measure zero case}
	  
	  Let us recall the following notion of small sets from \cite{BarJu}.
	  \begin{definition}
			We call $A\se 2^\omega$ a small set if there is a partition $\{I_n: \;n\in\w\}$ of $\omega$ into intervals and a collection $(J_n:\; n\in \w)$ such that $J_n\se 2^{I_n}$ for $n\in\w$ with $\sum_{n\in\w}\frac{|J_n|}{2^{|I_n|}}<\infty$ and
			\[
				A\se\{x\in 2^\omega: (\existinfty n\in\w)(x\rest I_n\in J_n)\}.
			\]
		\end{definition}
		Note that we use the simpler version of the definition using a partition of $\omega$ into intervals rather than any finite sets. However it is sufficient for the purpose of this paper, since we use the fact that each null set is a union of two small sets, even if the partition is into intervals \cite[Theorem 2.5.7]{BarJu}. The differences and similarities between these definitions (among other things) were discussed in \cite{BarSheSmall}.
		
		\begin{lemma}\label{lemma for small sets}
			For every small set $F\se 2^\omega$ and every Silver tree $T$ there is a Silver tree $T'\se T$ such that for every $n\in\omega$
			\[
			F+\underbrace{[T']+[T']+...+[T']}_{n\ti{--times}}\text{ is small.}
			\]
		\end{lemma}
		\begin{proof}
		  Let $F$ be small with an associated partition $\{I_n:\; n\in\omega\}$ and a collection of sequences $\{J_n: \; n\in\omega\}$. Let $T$ be a Silver tree with associated $A$ and $x_T$. Set $A'\se A$ infinite such that $|A'\cap I_n|\leq 1$ for any $n\in\omega$. Set new sequences of allowed patterns $J'_n=(J_n+x_T\rest I_n)\cup J_n$ and $J''_n=J'_n+(00\dots010\dots))$ ($1$ on position $i\in A\cap I_n$). See that the set
		  \[
		    \{x\in 2^\w:\; (\existinfty n\in \w)(x\rest I_n\in J''_n)\}
		  \]
		  is small since $|J''_n|\le 4|J_n|$.
		\end{proof}
		
		\begin{theorem}\label{Null Silver}
			For every null set $F\se 2^\omega$ and every Silver tree $T$ there is a Silver tree $T'\se T$ such that for every $n\in\omega$
			\[
				F+\underbrace{[T']+[T']+...+[T']}_{n\ti{--times}}\in\Nn.
			\]
		\end{theorem}
		\begin{proof}
		  Let $F$ be a null set and $F_1$, $F_2$ be small sets such that $F=F_1\cup F_2$. Let $T$ be a Silver tree. Then by Lemma \ref{lemma for small sets} there exists a Silver tree $T'\se T$ such that 
		  \[
				F_1+\underbrace{[T']+[T']+...+[T']}_{n\ti{--times}}  \text{ is small}.
			\]
			Again by Lemma \ref{lemma for small sets} for $F_2$ and $T'$ there exists a Silver tree $T''$ such that
			\[
			  F_2+\underbrace{[T'']+[T'']+...+[T'']}_{n\ti{--times}} \text{ is small}.
			\]
		  Then $F+\underbrace{[T'']+[T'']+...+[T'']}_{n\ti{--times}}$ is null as a union of two small sets.
		\end{proof}
		
		The above result allow us to obtain the thesis of Theorem \ref{Luziny t0} for $v_0$.
		\begin{corollary}\label{coro v0}
		  Let $\c$ be a regular cardinal. Then for every generalized Luzin set $L$ and generalized Sierpiński set $S$ we have $L+S\in v_0$.
		\end{corollary}
		
    The following result may seem out of context; however, it will be useful in the proof of Theorem \ref{null perfect}.
    \begin{lemma}\label{lemat analityczny}
		  Let $\sum_{n\in\omega}a_n=s<\infty$, where $a_n>0$ for all $n\in\omega$. Then there is a nondecreasing sequence of natural numbers $(k_n)_{n\in\omega}$ such that for each $b\in\omega$
		  \[
		    \sum_{n\in\omega}(2^b)^{k_n}a_n<2^{b^2}s.
		  \]
		\end{lemma}
		\begin{proof}
		  Let $n_0=0$ and $n_j=\min\{n\in\omega:\; \sum_{i\ge n}a_i<\frac{s}{2^{j^2}}\}$ and for all $n, j\in\omega$ set $k_n=j \iff n_j\le n<n_{j+1}$. Then for any $b\in\omega$
		  \begin{align*}
		    \sum_{n\in\omega}(2^b)^{k_n}a_n=\sum_{n<n_b}(2^b)^{k_n}a_n+\sum_{n\ge n_b}(2^b)^{k_n}a_n< (2^b)^{b-1}s+\sum_{i\in\omega}2^{b(b+i)}\frac{s}{2^{(b+i)^2}}<
		    \\
		    <2^{b(b-1)}s+2s\le2^{b^2}s.
		  \end{align*}
		\end{proof}
		
		\begin{theorem}\label{null perfect}
		  For every small set $A\in 2^\omega$ and every (uniformly) perfect tree $T$ there exists (uniformly) perfect tree $T'\se T$ such that
		  \[
				A+\underbrace{[T']+[T']+...+[T']}_{n\ti{--times}} \text{ is small}.
		  \]
		\end{theorem}
		\begin{proof}
		  
		  Let $F$ be a small set with associated partition $\{I_n:\; n\in\w\}$ of $\w$ into intervals and allowed sets of sequences $\{J_n:\; n\in\w\}$. Let $T$ be a perfect tree (the proof for a uniformly perfect tree is similar). For $a_n=\frac{|J_n|}{2^{|I_n|}}$ let $(k_n:\; n\in\w)$ be as in the Lemma \ref{lemat analityczny}. We can find a perfect tree $T'\se T$ such that for each $\sigma\in T'$ with $\dom(\sigma)=\max I_n$ the set $\{\tau\in \splitt (T'):\; \tau\se\sigma\}$ has cardinality at most $k_n$. Notice that $T'\rest I_n=\{\sigma\rest I_n:\; \sigma\in T'\}$ has at most $2^{k_n}$ elements. Set for any $b\in\w$
		  \[
		    J^b_n=J_n+\underbrace{(T'\rest I_n)+(T'\rest I_n)+\dots+(T'\rest I_n)}_{b\ti{--times}}.
		  \]
		  Hence, for every $b\in\w$
		  \[
		    A+\underbrace{[T']+[T']+...+[T']}_{b\ti{--times}}\se \{x\in 2^\omega: (\existinfty n\in\omega)(x\rest I_n\in J^b_n)\}.
		  \]
		  See that $\frac{|J^b_n|}{2^{|I_n|}}\le a_n(2^{k_n})^b$ and thanks to Lemma \ref{lemat analityczny} the series $\sum_{n\in\omega}\frac{|J^b_n|}{2^{|I_n|}}$ is convergent for every $b\in\omega$. Therefore $A+\underbrace{[T']+[T']+...+[T']}_{b\ti{--times}}$ is small for every $b\in\w$.
		\end{proof}
		
		Arguments similar to the ones used in Theorem \ref{Null Silver} yield the following result.
		
		\begin{corollary}
		  For every null set $A\in 2^\omega$ and every (uniformly) perfect tree $T$ there exists (uniformly) perfect tree $T'\se T$ such that
		  \[
				A+\underbrace{[T']+[T']+...+[T']}_{n\ti{--times}}\in \Nn.
		  \]
		\end{corollary}
		
		Relaying on the above result we can obtain the thesis of Theorem \ref{Luziny t0} regarding $u_0$.
		\begin{corollary}
		  Let $\c$ be a regular cardinal. Then for every generalized Luzin set $L$ and generalized Sierpiński set $S$ we have $L+S\in u_0$.
		\end{corollary}
	  
	  Similarly to Theorem \ref{meager splitting} we have a weaker result in the case of $\Nn$ and splitting trees.
	  \begin{theorem}\label{Splitting null}
			For every null set $F\se 2^\omega$ there is a splitting tree $T$ such that for every $n\in\omega$
			\[
				F+\underbrace{[T]+[T]+...+[T]}_{n\ti{--times}}\in\Nn.
			\]
		\end{theorem}
		\begin{proof} 
			Let $F_1, F_2$ be small sets such that $F=F_1\cup F_2$. For $j=1,2$ let  $\{I^j_n:\; n\in\w\}$  be  partitions of $\w$ into intervals and let $\{J^j_n:\; n\in\w\}$ be sets of allowed sequences associated with $F_j$. Without loss of generality (see \cite[Theorem 2.5.8]{BarJu}) we can assume that for every $j$ $I_n^j\cap I_m^{3-j}\ne\0$ for at most two $m$'s. Let $A$ be an  infinite subset of $\w$ such that $|A\cap I^j_n|\le 1$ for every $n\in\w, j\in\{1,2\}$. Let $\{K_n:\; n\in\omega\}$ be a partition of $\w$ into intervals such that
			\begin{align*}
			  I^1_n&=K_{2n}\cup K_{2n+1},
			  \\
			  I^2_n&=K_{2n+1}\cup K_{2n+2}.
			\end{align*}
			Define
			\[
			T=\{\sigma\in 2^{<\omega}:\; (\forall n\in\w)(\sigma\rest (K_n\bez A)\se\mbb{0}\rest (K_n\bez A) \lor \sigma\rest (K_n\bez A)\se\mbb{1}\rest (K_n\bez A))\}
			\]
			and notice that $T$ is a splitting tree. Set
			\begin{align*}
			  {J^1_n}'&=J^1_n+\{x\rest K_{2n}\cup y\rest K_{2n+1}:\; x,y\in\{\mbb{0}, \mbb{1}\}\}+\{\mbb{0}\rest I^1_n, \chi_A\rest I^1_n\},
			  \\
			  {J^2_n}'&=J^2_n+\{x\rest K_{2n+1}\cup y\rest K_{2n+2}:\; x,y\in\{\mbb{0}, \mbb{1}\}\}+\{\mbb{0}\rest I^2_n, \chi_A\rest I^2_n\}.
			\end{align*}
			For $j=1,2$ let
			\[
			  F_j'=\{x\in 2^\w:\; (\existinfty n\in\w)(x\rest I^j_n\in {J^j_n}')\}.
			\]
			Clearly these sets are small and 
			\[
				F_j+\underbrace{[T]+[T]+...+[T]}_{n\ti{--times}}\se F_j'.
			\]
			Hence, the proof is complete.
		\end{proof}
		
		Again, similarly to the case of $\Mm$, we may pose the following question.
		\begin{problem}
		  Let $F$ be a null subset of $2^\w$. Is it true that for every splitting tree $T$ there exists a splitting tree $T'\se T$ such that
		  \[
		    F+\underbrace{[T']+[T']+...+[T']}_{\tn{n-times}}\in\Nn?
		  \]
		\end{problem}
	  
	\subsection*{Intersection of measure and category}
	
	  The above theorems for measure and category allow us to obtain the following result.
	
	  \begin{theorem}
	    Let $F\in\Mm\cap \Nn$ and let $T$ be a perfect (resp. uniformly perfect or Silver) tree. Then there is a perfect (resp. uniformly perfect or Silver) tree $T'\se T$ such that
	    \[
	      F+\underbrace{[T']+[T']+...+[T']}_{n\ti{--times}}\in \Mm\cap\Nn.
	    \]
	  \end{theorem}
	  
	  The proof is similar to the proof of Theorem \ref{Null Silver}.
	
	\subsection*{Closed measure zero case}
	
		The following can be deciphered from the characterization of $\Ee$ from \cite[Lemma 2.6.3]{BarJu}.
		
		\begin{lemma}
		  For every set $E\in \Ee$ there is a partition $\{I_n:\; n\in\omega\}$ of $\omega$ to intervals and sets of finite sequences $J_n\se 2^{I_n}$, $n\in \omega$, such that $\sum_{n\in\omega} \frac{1}{2^{|I_n|}}<\infty$, for each $n\in\omega$ $\frac{|J_n|}{2^{I_n}}\le \frac{1}{2^n}$, and
		  \[
		    E\se \{x\in 2^\omega:\; (\foralmostall n)(x\rest I_n\in J_n)\}.
		  \]
		\end{lemma}
		
		Below is a simpler characterization of $\Ee$.
		
		\begin{lemma}
		  For every set $E\in \Ee$ there is a partition $\{I_n:\; n\in\omega\}$ of $\omega$ to intervals and sets of finite sequences $J_n\se 2^{I_n}$, $n\in \omega$, such that $\frac{|J_n|}{2^{|I_n|}}\leq \frac{1}{2}$ for each $n\in\omega$ and
		  \[
		    E\se \{x\in 2^\omega:\; (\foralmostall n)(x\rest I_n\in J_n)\}.
		  \]
		\end{lemma}
		\begin{proof}
		  Let $E\in\Ee$ and assume that $E\se \bigcup_{n\in\omega}F_n$, where $(F_n:\; n\in\omega)$ is an ascending sequence of closed null subsets of $2^\omega$. We will construct desired partition and sets of finite sequences inductively.
		  
		  Step $0$. Since $F_0$ is closed and null, there is an open set $U_0\es F_0$ satisfying $\lambda(U_0)<\frac{1}{2}$. Then there are finite sequences $\sigma^0_{k}\in 2^{<\omega}$, $k\in \omega$, such that $U_0=\bigcup_{k\in\omega}[\sigma^0_k]$. $F_0$ is compact, therefore there is $k_0\in\omega$ such that $F_0\se \bigcup_{k<k_0}[\sigma^0_k]$. Set $I_0=\bigcup_{k<k_0}\dom(\sigma^0_k)$ and $J_0=\{\sigma\in 2^{I_0}:\; (\exists k<k_0)(\sigma^0_k\se \sigma)\}$. Clearly, $\frac{|J_0|}{2^{|I_0|}}= \lambda(\bigcup_{k<k_0}[\sigma^0_k])<\frac{1}{2}$.
		  
		  At the step $n+1$ consider the set 
		  \[
		    \widetilde{F}_{n+1}=\{x\in 2^\omega:\; (\exists y\in F_{n+1}) (x\rest (\omega\bez \bigcup_{k\le n}I_k)=y\rest(\omega\bez \bigcup_{k\le n}I_k)\}).
		  \]
		See that $\widetilde{F}_{n+1}=F_{n+1}+2^{\bigcup_{k\le n}I_k}$, hence it is compact and null. There exists open $U_{n+1}\es \widetilde{F}_{n+1}$, $\lambda(U_{n+1})\le\frac{1}{2}$, and finite sequences $\sigma^{n+1}_{k}\in 2^{<\omega}$, $k\in \omega$, such that $U_{n+1}\se\bigcup_{k\in\omega}[\sigma^{n+1}_k]$. By compactness of $\widetilde{F}_{n+1}$ there is $k_{n+1}\in\omega$ such that $\widetilde{F}_{n+1}\se\bigcup_{k<k_{n+1}}[\sigma^{n+1}_{k}]$. Set 
		\begin{align*}
		  I_{n+1}&=\bigcup_{k<k_{n+1}}\dom(\sigma^{n+1}_k)\bez\bigcup_{k\le n}I_{k},
		  \\
		  J_{n+1}&=\{\sigma\in 2^{I_{n+1}}:\; (\exists k<k_{n+1})(\sigma^{n+1}_k\rest I_{n+1}=\sigma\rest I_{n+1})\}.
		\end{align*}
		Note that $\frac{|J_{n+1}|}{2^{|I_{n+1}|}}\le \lambda(U_{n+1})=\frac{1}{2}$. This completes the construction.
		
		To see that $E\se \{x\in 2^\omega:\; (\foralmostall n)(x\rest I_n\in J_n)\}$ let $x\in E$. Then there is $n_0\in\omega$ such that $x\in F_{n}$ for $n\ge n_0$. Then for all $n$ $x\in U_n\se \bigcup_{k<k_n}[\sigma^{n}_k]$, hence $x\rest I_n\in J_n$. The proof is complete.
    \end{proof}

    \begin{theorem}\label{Silver E}
			For every set $E\in\Ee$ and every Silver tree $T$ there is a Silver tree $T'\se T$ such that for every $n\in\omega$
			\[
				E+\underbrace{[T']+[T']+...+[T']}_{n\ti{--times}}\in\Ee.
			\]
		\end{theorem}
		\begin{proof}
		  Let $E\in \Ee$ with associated partition $\{I_n:\; n\in\omega\}$ of $\w$ into intervals, sets of sequences $J_n\se 2^{I_n}$, $n\in\w$ and let $T$ be a Silver tree with parameters $x_T\in 2^\w$ and $A\in [\w]^\w$. Set
		  \begin{align*}
		    I'_{n}&=I_{3n}\cup I_{3n+1}\cup I_{3n+2},
		    \\
		    A'&=\{a\in A:\; (\forall n\in\w)(a\in I'_n\to \neg (\exists b<a)(b\in I'_n\cap A))\},
		    \\
		    J'_n&=(J_{3n}\times J_{3n+1}\times J_{3n+2})+\{\mbb{0}\rest I'_n, x_T\rest I'_n, \chi_{A'}\rest I'_n, (x_T+\chi_{A'})\rest I'_n\}.
		  \end{align*}
		  Notice that $|A'|=\aleph_0$ and $|A'\cap I'_n|\leq 1$. Here $\chi_{A'}$ denotes the characteristic function of $A'$, i.e. $\chi_{A'}=A'\times \{1\}\cup (\w\bez A')\times\{0\}$. Set
		  \begin{align*}
		    T'&=\{\sigma\in 2^{<\w}:\; (\forall n\in\dom(\sigma))(n\notin A'\to \sigma(n)=x_T(n))\},
		    \\
		    E'&=\{x\in 2^\w:\; (\foralmostall n\in\omega)(x\rest I'_n\in J'_n)\}.
		  \end{align*}
      Clearly $T'$ is a Silver tree and $E'\in\Ee$. To see the latter notice that
      \[
        \frac{|J'_n|}{2^{|I'_n|}}\le \frac{|J_{3n}|\cdot |J_{3n+1}|\cdot |J_{3n+2}|\cdot 4}{2^{|I_{3n}|}2^{|I_{3n+1}|}2^{|I_{3n+2}|}}\le \left(\frac{1}{2}\right)^3\cdot 4=\frac{1}{2}.
      \]
      Now, let $x\in E$. There exists $N$ such that $x\rest I_n\in J_n$ for $n\ge N$. Let $K=\lceil N/3\rceil$. Then $x\rest I'_n\in J_{3n}\times J_{3n+1}\times J_{3n+2}$ for $n\ge K$. Let $y\in \underbrace{[T']+[T']+\dots + [T']}_{k-times}$. Then for every $n\in\w$
      \[
        y\rest I'_n\in \{\mbb{0}\rest I'_n, x_T\rest I'_n, \chi_{A'}\rest I'_n, (x_T+\chi_{A'})\rest I'_n\}.
      \]
      Hence, $x+y\in E'$.
		\end{proof}
		
		\begin{theorem}
			For every set $E\in\Ee$ and every (uniformly) perfect tree $T$ there is a (uniformly) perfect tree $T'\se T$ such that for every $n\in\omega$
			\[
				E+\underbrace{[T']+[T']+...+[T']}_{n\ti{--times}}\in\Ee.
			\]
		\end{theorem}
		\begin{proof}
		  Let $E\in \Ee$ with associated partition $\{I_n:\; n\in\omega\}$ of $\w$ into intervals, sets of sequences $J_n\se 2^{I_n}$, $n\in\w$ and let $T$ be a perfect tree (the proof for a uniformly perfect is almost identical). Set a sequence of integers $(k_n: n\ge -1)$ given recursively by $k_{-1}=-1, k_n=k_{n-1}+(2^n)^{n+1}$ for $n\ge 0$. For every $n\in\w$ set
		  \[
		    I'_n=\bigcup\{I_k:\; k_{n-1}<k\le k_n\}.
		  \]
		  Let $T'\se T$ be a perfect tree such that for each $\sigma\in T', |\sigma|=\max I'_n,$ the set $\{\tau\in \splitt (T'):\; \tau\se\sigma\}$ has cardinality at most $n$. For every $n\in\w$ let
		  \[
		    J'_n=\left(\prod_{k=k_{n-1}+1}^{k_n}J_k\right)+\bigcup_{j\le n}\underbrace{\{\sigma\rest I'_n:\; \sigma\in T'\}+\dots+\{\sigma\rest I'_n:\; \sigma\in T'\}}_{j-times}.
		  \]
      Notice that
      \[
        \frac{|J'_n|}{2^{|I'_n|}}\le \left(\prod_{k=k_{n-1}+1}^{k_n}\frac{|J_k|}{2^{|I_k|}}\right)\cdot (1+(2^n)^{1}+\dots+(2^n)^{n})\le \left(\frac{1}{2}\right)^{k_n-k_{n-1}}((2^n)^{n+1}-1)=\frac{1}{2}.
      \]
      It follows that the set
      \[
        E'=\{x\in 2^\w:\; (\foralmostall n\in\omega)(x\rest I'_n\in J'_n)\}\in \Ee.
      \]
    From the definition of sets $J'_n$ it is clear that
    \[
      E+\underbrace{[T']+[T']+\dots [T']}_{j-times}\se E'
    \]
    for each $j\in\omega$.
  \end{proof}
  
  \begin{theorem}
			For every set $F\in\Ee$ there is a splitting tree $T$ such that for every $n\in\omega$
			\[
				F+\underbrace{[T]+[T]+...+[T]}_{n\ti{--times}}\in\Ee.
			\]
  \end{theorem}
	\begin{proof} 
	  We proceed similarly to the proof of Theorem \ref{Silver E}. Let $E\in \Ee$ with associated partition $\{I_n:\; n\in\omega\}$ of $\w$ into intervals, sets of sequences $J_n\se 2^{I_n}$, $n\in\w$. Set
		  \begin{align*}
		    I'_{n}&=I_{3n}\cup I_{3n+1}\cup I_{3n+2},
		    \\
		    A'&=\{a\in A:\; (\forall n\in\w)(a\in I'_n\to \neg (\exists b<a)(b\in I'_n\cap A))\},
		    \\
		    J'_n&=(J_{3n}\times J_{3n+1}\times J_{3n+2})+\{\mbb{0}\rest I'_n, \mbb{1}\rest I'_n, \chi_{A'}\rest I'_n, (\mbb{1}-\chi_{A'})\rest I'_n\}.
		  \end{align*}
		  Set
		  \begin{align*}
		    T&=\{\sigma\in 2^{<\omega}:\; (\forall n\in\w)(\sigma\rest (I'_n\bez A)\se\mbb{0}\rest (I'_n\bez A) \lor \sigma\rest (I'_n\bez A)\se\mbb{1}\rest (I'_n\bez A))\},
		    \\
		    E'&=\{x\in 2^\w:\; (\foralmostall n\in\omega)(x\rest I'_n\in J'_n)\}.
		  \end{align*}
		  The proof that $T$ is splitting essentially was done in Theorem \ref{Splitting null} and the proof of $E'\in \Ee$ and $E+\underbrace{[T]+[T]+\dots+[T]}_{n-times}\se E'$ follows a similar pattern seen in Theorem \ref{Silver E}.
	\end{proof}
		
  \begin{problem}
		Let $F\in \Ee$. Is it true that for every splitting tree $T$ there exists a splitting tree $T'\se T$ such that
		\[
		  F+\underbrace{[T']+[T']+...+[T']}_{\tn{n-times}}\in\Ee?
		\]
	\end{problem}
	\printbibliography

\end{document}